\documentclass[reqno,12pt]{amsart} 
\usepackage{vmargin}
\setpapersize{A4}

\usepackage{amsmath}

\usepackage{amsfonts}   

\usepackage{amssymb}      

\usepackage{eufrak}      

\usepackage{amscd}      
\usepackage{amsthm}      

\usepackage{tikz,amsmath}
\usetikzlibrary{arrows}
\usepackage{tikz}
\usetikzlibrary{matrix,calc}

\usepackage{epsfig}      

\usepackage{amstext}      

\usepackage{graphicx}

\usepackage[all,line,dvips]{xy} 
\CompileMatrices 

\newcommand{\Ad}{\operatorname{Ad}}

\newcommand{\Aut}{\operatorname{Aut}}

\newcommand{\Span}{\operatorname{Span}}

\newcommand{\Sp}{\operatorname{Sp}}

   \theoremstyle{plain}
   \newtheorem{thm}{Theorem}[section]
   
   \newtheorem{lemma}[thm]{Lemma}  
   \newtheorem{cor}[thm]{Corollary}
   \theoremstyle{definition}

   \newtheorem{example}[thm]{Example}
   \theoremstyle{remark}
   \newtheorem{q}[thm]{Question}

\usepackage{mdframed,xcolor}

\definecolor{mybgcolor}{gray}{0.8}
\definecolor{myframecolor}{rgb}{.647,.129,.149}

\mdfdefinestyle{mystyle}{
  usetwoside=false,
  skipabove=0.6em plus 0.8em minus 0.2em,
  skipbelow=0.6em plus 0.8em minus 0.2em,
  innerleftmargin=.25em,
  innerrightmargin=0.25em,
  innertopmargin=0.25em,
  innerbottommargin=0.25em,
  leftmargin=-.75em,
  rightmargin=-0em,
  topline=false,
  rightline=false,
  bottomline=false,
  leftline=false,
  backgroundcolor=mybgcolor,
  splittopskip=0.75em,
  splitbottomskip=0.25em,
  innerleftmargin=0.5em,
  leftline=true,
  linecolor=myframecolor,
  linewidth=0.25em,
}

\newmdenv[style=mystyle]{important}

\usepackage{lipsum}

   \numberwithin{equation}{section}

        \date{\today}

\title[Factor types]{The factor type of conservative KMS weights on graph $C^*$-algebras}
\author{Klaus Thomsen}

\date{\today}

\email{matkt@math.au.dk}
\address{Department of Mathematics, Aarhus University, Ny Munkegade, 8000 Aarhus C, Denmark}

\begin{document}

\maketitle


\bigskip

\section{Introduction} 

In certain quantum statistical models the observables are represented by the elements of a $C^*$-algebra and the time-evolution by a one-parameter group of automorphisms on the $C^*$-algebra, \cite{BR}. In such models the equilibrium states are characterized by the KMS-condition, \cite{HWH}, and much work has been devoted to the study of these KMS-states. As a result there are now large classes of $C^*$-algebras and one-parameter groups for which the structure of the equilibrium states is completely understood. An extremal KMS-state, and more generally also an extremal KMS-weight, gives rise in a canonical way to a representation of the $C^*$-algebra which generates a von Neumann algebra factor. Some of the papers on KMS-states have included a determination of this factor for the extremal KMS-states. This is the case in \cite{EFW}, \cite{BC}, \cite{O}, \cite{BF} \cite{IKW}, \cite{N}, \cite{LN}, \cite{Y1}, \cite{LLNSW}, \cite{CPPR}, \cite{KW}, \cite{Th1}, \cite{Th2}, \cite{Th4},\cite{Y2} and \cite{I}. \footnote{Despite some effort to make this a complete list, presumably it is not.} The factor types give a natural way to distinguish between KMS states very similar to the type classification of non-singular transformations in ergodic theory. The purpose with this paper is to determine the factor types for a class of KMS-weights and KMS-states that arise from generalized gauge actions on simple graph $C^*$-algebras. In \cite{Th3} the author studied such KMS-weights and among other things it was shown that they can be naturally divided into three classes depending on the properties of the measure they induce on the unit space of the groupoid underlying the $C^*$-algebra. For the KMS-weights we consider here the measure is concentrated on the space of infinite paths in the graph and it is conservative with respect to the action of the shift. We call them therefore conservative KMS-weights. We determine their factor type when the graph $C^*$-algebra is simple by calculating the $\Gamma$-invariant of Connes for the associated factors. As shown in \cite{Th3} there is a bijective correspondence between the rays of KMS-weights and the KMS-states on a corner defined by a vertex in the graph, at least for the cases we consider in this paper, and we show therefore first that the $\Gamma$-invariant for the factor defined by an extremal KMS-weight is the same as the $\Gamma$-invariant for the factor obtained from the corresponding KMS-state on the corner. This is not surprising, but it is helpful because it shows that our calculation of the 
$\Gamma$-invariant covers two cases.

\emph{Acknowledgement} I am grateful to Johannes Christensen for discussions and help to eliminate mistakes. The work was supported by the DFF-Research Project 2 `Automorphisms and Invariants of Operator Algebras', no. 7014-00145B.

\section{The setting}

\subsection{Factors from KMS-weights}

Let $A$ be a $C^*$-algebra and $A_{+}$ the convex cone of positive
elements in $A$. A \emph{weight} on $A$ is a map $\psi : A_{+} \to
[0,\infty]$ with the properties that $\psi(a+b) = \psi(a) + \psi(b)$
and $\psi(\lambda a) = \lambda \psi(a)$ for all $a, b \in A_{+}$ and all
$\lambda \in \mathbb R, \ \lambda > 0$. By definition $\psi$ is
\emph{densely defined} when $\left\{ a\in A_{+} : \ \psi(a) <
  \infty\right\}$ is dense in $A_{+}$ and \emph{lower semi-continuous}
when $\left\{ a \in A_{+} : \ \psi(a) \leq \alpha \right\}$ is closed
for all $\alpha \geq 0$. We will use \cite{Ku}, \cite{KV1} and \cite{KV2} as our source of information on weights, and as in \cite{KV2} we say that a weight is \emph{proper}
when it is non-zero, densely defined and lower semi-continuous. 

Let $\psi$ be a proper weight on $A$. Set $\mathcal N_{\psi} = \left\{ a \in A: \ \psi(a^*a) < \infty
\right\}$ and note that 
\begin{equation*}\label{f3}
\mathcal N_{\psi}^*\mathcal N_{\psi} = \Span \left\{ a^*b : \ a,b \in
  \mathcal N_{\psi} \right\}
\end{equation*} 
is a dense
$*$-subalgebra of $A$, and that there is a unique well-defined linear
map $\mathcal N_{\psi}^*\mathcal N_{\psi} \to \mathbb C$ which
extends $\psi : \mathcal N_{\psi}^*\mathcal N_{\psi} \cap A_+ \to
[0,\infty)$. We denote also this densely defined linear map by $\psi$.

Let $\alpha : \mathbb R \to \Aut A$ be a point-wise
norm-continuous one-parameter group of automorphisms on
$A$. Let $\beta \in \mathbb R$. Following \cite{C} we say that a proper weight
$\psi$ on $A$ is a \emph{$\beta$-KMS
  weight} for $\alpha$ when
\begin{enumerate}
\item[i)] $\psi \circ \alpha_t = \psi$ for all $t \in \mathbb R$, and
\item[ii)] for every pair $a,b \in \mathcal N_{\psi} \cap \mathcal
  N_{\psi}^*$ there is a continuous and bounded function $F$ defined on
  the closed strip $D_{\beta}$ in $\mathbb C$ consisting of the numbers $z \in \mathbb C$
  whose imaginary part lies between $0$ and $\beta$, and is
  holomorphic in the interior of the strip and satisfies that
$$
F(t) = \psi(a\alpha_t(b)), \ F(t+i\beta) = \psi(\alpha_t(b)a)
$$
for all $t \in \mathbb R$. \footnote{Note that we apply the
  definition from \cite{C} for the action $\alpha_{-t}$
  in order to use the same sign convention as in \cite{BR}, for example.}
\end{enumerate}   
A $\beta$-KMS weight $\psi$ with the property that 
$$
\sup \left\{ \psi(a) : \ 0 \leq a \leq 1 \right\} = 1
$$
will be called a \emph{$\beta$-KMS state}. The following is Theorem 2.4 in \cite{Th3}.

\begin{thm}\label{h5}  Let $\alpha : \mathbb R \to \Aut A$ be a point-wise
norm-continuous one-parameter group of automorphisms on a $C^*$-algebra
$A$. Let $p$ be a projection in the fixed point algebra of
$\alpha$ such that $p$ is full in $A$. For all $\beta \in \mathbb R$ the map
$$
\psi \mapsto  \psi(p)^{-1} \psi|_{pAp}
$$
is a bijection between the set of rays of $\beta$-KMS weights for $\alpha$ and
the $\beta$-KMS states for the restriction of $\alpha$ to $pAp$.
\end{thm}

Given a proper weight $\psi$ on a $C^*$-algebra $A$ there is a GNS-type
construction consisting of a Hilbert space $H_{\psi}$, a linear map
$\Lambda_{\psi} : \mathcal N_{\psi} \to H_{\psi}$ with dense range and
a non-degenerate representation $\pi_{\psi}$ of $A$ on $H_{\psi}$ such that
\begin{enumerate}
\item[$\bullet$] $\psi(b^*a) = \left<
    \Lambda_{\psi}(a),\Lambda_{\psi}(b)\right>, \ a,b \in \mathcal
  N_{\psi}$, and
\item[$\bullet$] $\pi_{\psi}(a)\Lambda_{\psi}(b) = \Lambda_{\psi}(ab), \
  a \in A, \ b \in \mathcal N_{\psi}$,
\end{enumerate} 
cf. \cite{Ku}, \cite{KV1}, \cite{KV2}.
A $\beta$-KMS weight $\psi$ on $A$ is \emph{extremal} when the only
$\beta$-KMS weights $\varphi$ on $A$ with the property that
$\varphi(a) \leq \psi(a)$ for all $a \in A_+$ are scalar
multiples of $\psi$, viz. $\varphi = s\psi$ for some $s > 0$.

The following is Lemma 4.9 in \cite{Th1}.

\begin{lemma}\label{k20} Let $A$ be a separable $C^*$-algebra and
  $\alpha$ a pointwise norm-continuous one-parameter group of automorphisms on
  $A$. Let $\psi$ be an extremal $\beta$-KMS weight for $\alpha$. Then
  $\pi_{\psi}(A)''$ is a factor.
\end{lemma}

It is shown in Section 2.2 of \cite{KV1} that a $\beta$-KMS weight $\psi$ extends to a normal semi-finite faithful weight
$\tilde{\psi}$ on $\pi_{\psi}\left(A\right)''$ such that $\psi =
\tilde{\psi} \circ \pi_{\psi}$, and that the modular group on
$\pi_{\psi}(A)''$ corresponding to $\tilde{\psi}$ is the
one-parameter group $\theta$ on $\pi_{\psi}(A)''$ given by
$$
\theta_t = \tilde{\alpha}_{-\beta t} \ ,
$$
where $\tilde{\alpha}$ is the
$\sigma$-weakly continuous extension of $\alpha$ defined such that
$\tilde{\alpha}_t \circ \pi_{\psi} = \pi_{\psi} \circ \alpha_t$. By construction $\tilde{\alpha}_t = \Ad U_t$, where $U_t \in B(H_{\psi})$ is defined by
$$
U_t \Lambda_{\psi}(a) = \Lambda_{\psi}(\alpha_t(a)) \ .
$$
In the setting of Theorem \ref{h5}, let $(\pi_{\varphi},H_{\varphi}, \xi_{\varphi})$ be the GNS-representation of the state $\varphi$ on $pAp$ defined such that
$$
\varphi(x) = \psi(p)^{-1}\psi(x) \ .
$$ 
The modular automorphism group $\theta'$ on $\pi_{\varphi}(pAp)''$ corresponding to the vector state defined by $\xi_{\varphi}$ is given by
$$
\theta'_t\left(\pi_{\varphi}(pap)\right) = \pi_{\varphi}\left(\alpha_{-\beta t}(pap)\right) =\pi_{\varphi}\left(p\alpha_{-\beta t}(a)p\right) \ .
$$ 
\begin{lemma}\label{07-03-18} In the setting of Theorem \ref{h5} there is an $*$-isomorphism 
\begin{equation}\label{07-03-18a}
\pi_{\varphi}\left(pAp\right)'' \ \simeq \ \pi_{\psi}(p)\pi_{\psi}(A)''\pi_{\psi}(p)
\end{equation}
of von Neumann algebras which is equivariant with respect to $\theta$ and $\theta'$. 
\end{lemma}
\begin{proof} Let $q \in B(H_{\psi})$ be the orthogonal projection on $ \overline{\Lambda_{\psi}\left(p Ap\right)}$ and define a unitary 
$W : H_{\varphi} \to qH_{\psi}$
such that 
$$
W\pi_{\varphi}(x)\xi_{\varphi} = \psi\left(p\right)^{-\frac{1}{2}} \Lambda_{\psi}(x) 
$$ 
for $x \in pAp$. Conjugation by $W$ gives an isomorphism $\pi_{\varphi}\left(pAp\right)'' \simeq q \pi_{\psi}(A)''q$. Since $\left(qH_{\psi}, \pi_{\psi}, \frac{\Lambda_{\psi}(p)}{\sqrt{\psi(p)}} \right)$ is the GNS-triple of $\varphi$ it follows from Corollary 5.3.9 in \cite{BR} that $\Lambda_{\psi}(p)$ is separating for $\pi_{\psi}(p)\pi_{\psi}(A)''\pi_{\psi}(p)$. Since $q$ commutes with $\pi_{\psi}(p)\pi_{\psi}(A)''\pi_{\psi}(p)$ and $qH_{\psi}$ contains $\Lambda_{\psi}(p)$, the map
$m \mapsto mq$
is an isomorphism 
$$
\pi_{\psi}(p)\pi_{\psi}(A)''\pi_{\psi}(p) \to q \pi_{\psi}(A)''q \ .
$$ 
We obtain then the isomorphism \eqref{07-03-18a} as the composition of two isomorphisms, both of which are equivariant.
\end{proof}

When $M$ is a $\sigma$-finite von Neumann algebra factor every normal faithful semi-finite weight on $M$ comes together with a modular automorphism group $\theta = (\theta_t)_{t \in \mathbb R}$ and the Connes-invariant $\Gamma(M)$ is the intersection
$$
\Gamma(M) = \bigcap_q \Sp(qMq)  \ ,
$$
where we take the intersection over all projections $q \in M$ fixed by $\theta$, and $\Sp(qMq)$ denotes the Arveson spectrum of the restriction of $\theta$ to $qMq$. In more detail, $\Sp(qMq)$ is defined as follows. For $f \in L^1(\mathbb R)$ define a linear map
$\theta_f : qMq \to qMq$ such that
$$
\theta_f(a) = \int_{\mathbb R} f(t) \theta_t(a) \ dt \ .
$$  
Then
$$
\Sp (qMq) = \bigcap \left\{ Z(f) : \ f \in L^1(\mathbb R), \
  \theta_f(qMq) = \{0\} \right\} \ ,
$$
where
$$
Z(f) = \left\{ r \in \mathbb R : \ \int_{\mathbb R} e^{itr} f(t) \ dt
  \ = \ 0 \right\} \  .
$$
See \cite{Co1}. In particular, when $\psi$ is an extremal $\beta$-KMS weight on $A$ we can calculate $\Gamma\left(\pi_{\psi}(A)''\right)$ by using the automorphism group $\theta_t = \tilde{\alpha}_{-\beta t}$ and we get the following immediate corollary to Lemma \ref{07-03-18}.

\begin{cor}\label{22-03-18} In the setting of Lemma \ref{07-03-18}, 
$$\Gamma\left(\pi_{\varphi}\left(pAp\right)''\right) = \Gamma\left(\pi_{\psi}(A)''\right) \ .
$$
\end{cor}


\subsection{Generalized gauge actions on graph $C^*$-algebras}\label{sec3}

Let $\mathcal G$ be a countable directed graph with vertex set $\mathcal G_V$ and arrow set
$\mathcal G_{Ar}$. \label{GammaV}\label{GammaAr}For an arrow $a \in \mathcal G_{Ar}$ we denote by $s(a) \in \mathcal G_V$ its source and by
$r(a) \in \mathcal G_V$ its range. A vertex $v$ which does not emit any arrow is
a \emph{sink}, while a vertex $v$ which emits infinitely many arrows is called an \emph{infinite
  emitter}. The set of sinks and infinite emitters in  $\mathcal G$ is denoted by $V_{\infty}$. An \emph{infinite path} in $\mathcal G$ is an element
$p  \in \left(\mathcal G_{Ar}\right)^{\mathbb N}$ such that $r(p_i) = s(p_{i+1})$ for all
$i$.  A finite path $\mu = a_1a_2 \cdots a_n = (a_i)_{i=1}^n \in \left(\mathcal G_{Ar}\right)^n$ is
defined similarly. The number of edges in $\mu$ is its \emph{length}
and we denote it by $|\mu|$. A vertex $v \in \mathcal G_V$ will be considered as
a finite path of length $0$. We let $P(\mathcal G)$\label{PGamma} denote the (possibly empty) set of
infinite paths in $\mathcal G$ and $P_f(\mathcal G)$\label{PfGamma} the set of finite paths in $\mathcal G$. The set $P(\mathcal G)$ is a complete metric space when the metric is given by
\begin{equation}\label{25-02-18}
d(p,q) = \sum_{i=1}^{\infty} 2^{-i}\delta(p_i,q_i) \ ,
\end{equation}
where $\delta(a,a) = 0$ and $\delta(a,b) = 1$ when $a \neq b$. We
extend the source map to $P(\mathcal G)$ such that
$s(p) = s(p_1)$ when $p = \left(p_i\right)_{i=1}^{\infty} \in P(\mathcal G)$, and the
range and source maps to $P_f(\mathcal G)$ such that $s(\mu) = s(a_1)$ and $r(\mu)
= r(a_{|\mu|})$ when $|\mu|\geq 1$, and $s(v) = r(v) = v$ when $v\in \mathcal G_V$. Associated to the finite path $\mu \in P_f(\mathcal G)$ is the  cylinder set \label{Zmu}
$$
Z(\mu) = \left\{(p_i)_{i=1}^{\infty} \in P(\mathcal G) : \  
p_j = a_j, \ j = 1,2, \cdots, |\mu|\right\} 
$$
which is an open and closed set in $P(\mathcal G)$. In particular, when $\mu$ has length $0$ and hence is just a vertex $v$,
$$
Z(v) = \left\{ p \in P(\mathcal G) :  \  s(p) =v \right\} \  .
$$

The $C^*$-algebra $C^*(\mathcal G)$ of the graph $\mathcal G$ was introduced in this generality in \cite{BHRS} as the universal
$C^*$-algebra generated by a collection $S_a, a \in \mathcal G_{Ar}$, of partial
isometries and a collection $P_v, v \in \mathcal G_V$, of mutually orthogonal projections subject
to the conditions that
\begin{enumerate}
\item[1)] $S^*_aS_a = P_{r(a)}, \ \forall a \in \mathcal G_{Ar}$,
\item[2)] $S_aS_a^* \leq P_{s(a),} \ \forall a \in \mathcal G_{Ar}$,
\item[3)] $ \sum_{a \in F} S_aS_a^* \leq P_v$ for every finite subset $F \subseteq s^{-1}(v)$ and all $v \in
  \mathcal G_V$, and
\item[4)] $P_v = \sum_{a  \in s^{-1}(v)} S_aS_a^*, \ \forall v \in \mathcal G_V
  \backslash V_{\infty}$.
\end{enumerate} 
For a finite path $\mu = (a_i)_{i=1}^{|\mu|} \in P_f(\mathcal G)$ we set
$$
S_{\mu} = S_{a_1}S_{a_2}S_{a_3} \cdots S_{a_{|\mu|}} \ .
$$
The elements $S_{\mu}S_{\nu}^*, \mu,\nu \in P_f(\mathcal G)$, span a dense $*$-subalgebra $\mathcal A$ in $C^*(\mathcal G)$. The projections
$$
P_{\mu} = S_{\mu}S_{\mu}^*
$$ 
will play an important role in the following.

\begin{lemma}\label{nuclear} $C^*(\mathcal G)$ is a nuclear $C^*$-algebra and $\pi(C^*(\mathcal G))''$ is a hyperfinite von Neumann algebra for all non-degenerate representations $\pi$ of $C^*(\mathcal G)$.
\end{lemma}
\begin{proof} The nuclearity of $C^*(\mathcal G)$ follows from \cite{KP} when $\mathcal G$ is row-finite in the sense that $\# s^{-1}(v)< \infty$ for all $v \in \mathcal G_V$, and the general case follows then from \cite{DT}. The second statement is a wellknown consequence of the first and goes back to \cite{CE} and \cite{Co2}. 
\end{proof}

 We describe next the necessary and sufficient conditions which $\mathcal G$ must satisfy for $C^*(\mathcal G)$ to be simple. These conditions were identified by Szymanski in \cite{Sz}. A \emph{loop} in $\mathcal G$ is a finite path $\mu \in P_f(\mathcal G)$ of positive length such that $r(\mu) = s(\mu)$. We will say that a loop $\mu$ \emph{has an exit} then $\# s^{-1}(v) \geq 2$ for at least one vertex $v$ in $\mu$. A subset $H \subseteq \mathcal G_V$ is \emph{hereditary} when $a \in
\mathcal G_{Ar}, \ s(a) \in H \Rightarrow r(a) \in H$, and \emph{saturated} when 
$$
v \in \mathcal G_V\backslash V_{\infty}, \  
r(s^{-1}(v))  \subseteq H \ \Rightarrow \ v\in H \ .
$$ 
In the following we say that $\mathcal G$ is \emph{cofinal} when the only
non-empty subset of $\mathcal G_V$ which is both hereditary and
saturated is $\mathcal G_V$ itself.

\begin{thm}\label{Sz} (Theorem 12 in \cite{Sz}.) $C^*(\mathcal G)$ is simple if and only if $\mathcal G$ is cofinal and every loop in $\mathcal G$ has an exit.
\end{thm}

A function $F :  \mathcal G_{Ar} \to \mathbb R$ will be called a \emph{potential} on $\mathcal G$. Using it we can define a pointwise norm-continuous one-parameter group $\alpha^F = \left(\alpha^F_t\right)_{t \in \mathbb R}$  on $C^*(\mathcal G)$ such that \label{alphaF}
$$
\alpha^F_t(S_a) = e^{i F(a) t} S_a
$$
for all $a \in \mathcal G_{Ar}$ and 
$$
\alpha^F_t(P_v) = P_v
$$
for all $v \in \mathcal G_V$. An action of this sort is called a \emph{generalized gauge action}; the \emph{gauge action} itself being the one-parameter group corresponding to the constant function $F =1$. To describe the KMS-weights for a generalized gauge action, extend $F$ to a map $F : P_f(\mathcal G) \to \mathbb R$ such that $F(v) = 0$ when $v \in \mathcal G_V$, and 
$$
F(\mu) = \sum_{i=1}^n F(a_i)
$$
when $\mu = (a_i)_{i=1}^n  \in \left(\mathcal G_{Ar}\right)^n$. For $\beta \in \mathbb R$, define the matrix $A(\beta) = \left(A(\beta)_{v,w}\right)_{v,w \in \mathcal G_V}$ over $\mathcal G_V$ such that\label{Abeta}
$$
A(\beta)_{v,w} \ \ = \sum_{ a } e^{-\beta F(a)}  \ 
$$
where we sum over arrows $a \in \mathcal G_{Ar}$ with $s(a) = v$ and $r(a) = w$. As in \cite{Th3} we say that a non-zero non-negative vector $\psi = \left(\psi_v\right)_{v \in \mathcal G_V}$ is \emph{almost $A(\beta)$-harmonic} when 
\begin{itemize}
\item $\sum_{w \in \mathcal G_V} A(\beta)_{v,w} \psi_w \leq \psi_v, \ \forall v \in \mathcal G_V$, and
\item $\sum_{w \in \mathcal G_V} A(\beta)_{v,w} \psi_w = \psi_v, \ \forall v \in \mathcal G_V \backslash V_{\infty}$.
\end{itemize} 
The almost $A(\beta)$-harmonic vectors $\psi$ for which 
\begin{itemize}
\item $\sum_{w \in \mathcal G_V} A(\beta)_{v,w} \psi_w = \psi_v, \ \forall v \in \mathcal G_V,$
\end{itemize}
will be called \emph{$A(\beta)$-harmonic}. In particular, when $\mathcal G$ is row-finite without sinks an almost $A(\beta)$-harmonic vector is automatically $A(\beta)$-harmonic. An almost $A(\beta)$-harmonic vector which is not $A(\beta)$-harmonic will be said to be a \emph{proper almost $A(\beta)$-harmonic vector}. It was shown in \cite{Th3} that an almost $A(\beta)$-harmonic vector $\varphi$ gives rise to a $\beta$-KMS weight $W_{\varphi}$ characterized by the properties that $S_{\mu}^* \in \mathcal N_{W_{\varphi}}$ and
$$
W_{\varphi}\left(S_{\mu}S_{\nu}^*\right) = \begin{cases} 0, & \ \mu \neq \nu \\  e^{-\beta F( \mu)} \varphi_{r(\mu)}, & \ \mu = \nu  \end{cases}
$$
when $\mu, \nu \in P_f(\mathcal G)$, and that all \emph{gauge invariant} KMS-weights arise like this. Generally there can be KMS-weights that are not gauge invariant, but as shown in Proposition 5.6 of \cite{CT} all KMS-weights are gauge invariant when $C^*(\mathcal G)$ is simple. Therefore, in the case that concerns us here, all KMS-weights arise from almost $A(\beta)$-harmonic vectors. Borrowing terminology from harmonic analysis we say that an almost $A(\beta)$-harmonic vector $\varphi$ is \emph{minimal} when it only dominates multiples of itself, i.e. when every almost $A(\beta)$-harmonic vector $\varphi'$ with the property that $\varphi'_v \leq \varphi_v$ for all $v \in \mathcal G_V$ is of the form $\varphi' = \lambda \varphi$ for some $\lambda > 0$. Thus the minimal almost $A(\beta)$-harmonic vectors are those for which the corresponding $\beta$-KMS weight $W_{\varphi}$ is extremal.

The minimal almost $A(\beta)$-harmonic vectors can be subdivided in various ways. Here we shall consider three fundamental classes. The first class consists of the minimal proper almost $A(\beta)$-harmonic vectors. The other two classes consist of the $A(\beta)$-harmonic vectors and are distinguished by the properties of the measures they define on $P(\mathcal G)$. To explain this observe that by Lemma 3.7 in \cite{Th3} an $A(\beta)$-harmonic vector $\varphi$ defines a Borel measure $m_{\varphi}$ on $P(\mathcal G)$ such that
\begin{equation}\label{28-03-18b}
m_{\varphi}(Z(\mu)) = e^{-\beta F(\mu)}\varphi_{r(\mu)}
\end{equation}
for all $\mu \in P_f(\mathcal G)$. The Borel measures $m$ on $P(\mathcal G)$ that arise from $A(\beta)$-harmonic vectors in this way are characterized by the two properties that
\begin{itemize}
\item $m(Z(v)) < \infty$ for all $v \in \mathcal G_V$, and
\item $
m\left(\sigma(B\cap Z(a))\right) = e^{\beta F(a)} m( B \cap
Z(a))$ 
for every edge $a \in \mathcal G_{Ar}$ and every Borel subset $B$ of $P(\mathcal G)$. 
\end{itemize}
Here $\sigma$ denotes the shift on $P(\mathcal G)$, defined such that $\sigma(p)_i = p_{i+1}$. Non-zero Borel measures on $P(\mathcal G)$ with these two properties are called $e^{\beta F}$-conformal, \cite{Th5}. They are the measures that were called harmonic $\beta$-KMS measures in \cite{Th3}. Let $\mathcal G$ be a cofinal graph. As in \cite{Th1} and \cite{Th3} we say that a vertex $v \in \mathcal G_V$ is \emph{non-wandering} when there is a finite path $\mu \in P_f(\mathcal G)$ of positive length such that $v = s(\mu) = r(\mu)$. When the set $NW_{\mathcal G}$ of non-wandering vertexes is not empty it is a hereditary subset of $\mathcal G_V$, and together with the arrows 
$$
NW_{Ar} = \left\{ a \in \mathcal G_{Ar} : \ s(a) \in NW_{\mathcal G} \right\} \ 
$$
they constitute a strongly connected subgraph of $\mathcal G$ which we also denote by $NW_{\mathcal G}$, cf. Proposition 4.3 in \cite{Th3}. Set
$$
P(\mathcal G)_{rec} = \bigcap_{a \in NW_{Ar}} \left\{p \in P(\mathcal G): \ p_i = a \ \text{for infinitely many} \ i \right\} 
$$
and
$$
P(\mathcal G)_{wan} = \bigcap_{v \in \mathcal G_V} \left\{ p \in P(\mathcal G): \ \# \left\{ i \in \mathbb N: \ s(p_i) = v\right\} < \infty  \ \right\}  \ .
$$ 
We say that an $A(\beta)$-harmonic vector $\varphi$ is \emph{conservative} when $m_{\varphi}$ is concentrated on $P(\mathcal G)_{rec}$ and that $\varphi$ is \emph{dissipative} when $m_{\varphi}$ is concentrated on $P(\mathcal G)_{wan}$. When $NW_{\mathcal G}$ is empty, $P(\mathcal G)_{rec} = \emptyset$ and $P(\mathcal G)_{wan} = P(\mathcal G)$, and hence all $A(\beta)$-harmonic vectors are dissipative. To see that we have introduced a genuine dichotomy we need the following. For strongly connected graphs it is Theorem 4.10 in \cite{Th3}.

\begin{thm}\label{21-03-18x} Let $\mathcal G$ be a cofinal digraph such that $NW_{\mathcal G} \neq \emptyset$. Every $e^{\beta F}$-conformal measure $m$ is concentrated either on $P(\mathcal G)_{rec}$ or on $P(\mathcal G)_{wan}$, and
\begin{itemize}
\item $m$ is concentrated on $P(\mathcal G)_{rec}$ if and only if 
$$
\sum_{n=0}^{\infty} A(\beta)^n_{v,v} = \infty
$$
 for one, and hence all $v \in NW_{\mathcal G}$, and
\item $m$ is concentrated on $P(\mathcal G)_{wan}$ if and only if 
$$
\sum_{n=0}^{\infty} A(\beta)^n_{v,v} < \infty
$$ 
for one, and hence all $v \in NW_{\mathcal G}$.
\end{itemize}
\end{thm}
\begin{proof} Consider an $e^{\beta F}$-conformal measure $m$ on $P(\mathcal G)$.
For every $\mu \in P_f(\mathcal G)$, set 
$$
Z(\mu)P(NW_{\mathcal G}) \ = \ \left\{(p_i)_{i=1}^{\infty} \in Z(\mu): \ (p_i)_{i=|\mu|+1}^{\infty} \in  P(NW_{\mathcal G}) \ \right\} \ .
$$
Note that since $m$ is $e^{\beta F}$-conformal,
$$
m\left( Z(\mu)P(NW_{\mathcal G})\right) = e^{-\beta F(\mu)} m\left( 
P(NW_{\mathcal G}) \cap Z(r(\mu))\right) .
$$ 
Combined with the observation that
\begin{equation}\label{21-03-18}
P(\mathcal G) = \bigcup_{\mu \in P_f(\mathcal G)}  Z(\mu)P(NW_{\mathcal G})
\end{equation}
by Proposition 4.3 in \cite{Th3}, it follows that $m(P(NW_{\mathcal G})) \neq 0$. In short, no $e^{\beta F}$-conformal measure annihilates $P(NW_{\mathcal G})$. Therefore all conclusions follow from \eqref{21-03-18} above and Theorem 4.10 in \cite{Th3}. 

\end{proof}

It follows that when $\mathcal G$ is cofinal every minimal $A(\beta)$-harmonic vector is either
 \begin{itemize}
\item a proper almost $A(\beta)$-harmonic vector,
\item a conservative $A(\beta)$-harmonic vector or
\item a dissipative $A(\beta)$-harmonic vector.
\end{itemize}

In this paper we focus on the conservative minimal $A(\beta)$-harmonic vectors. We call the corresponding $\beta$-KMS weights \emph{conservative}. The adjective 'minimal' is superfluous in connection with conservative $\beta$-KMS weights because of the following

\begin{thm}\label{21-03-18c} Assume that $C^*(\mathcal G)$ is simple. There is a conservative $A(\beta)$-harmonic vector if and only if 
\begin{itemize}
\item $NW_{\mathcal G} \neq \emptyset$,
\item $\sum_{n=0}^{\infty} A(\beta)^n_{v,v} = \infty$, and 
\item $\limsup_n \left(A(\beta)^n_{v,v}\right)^{\frac{1}{n}} = 1$
\end{itemize} 
for one and hence all $v \in NW_{\mathcal G}$. When it exists it is unique up to multiplication by scalars and it is minimal.
\end{thm}
\begin{proof} As observed above there can not be any conservative $e^{\beta F}$-conformal measure on $P(\mathcal G)$ unless $NW_{\mathcal G} \neq \emptyset$. Therefore all statements follow by combining Theorem \ref{21-03-18x} above with Proposition 4.9 and Theorem 4.14 in \cite{Th3}.
\end{proof}

In particular, when $C^*(\mathcal G)$ is simple the existence of a conservative KMS-weight for $\alpha^F$ depends only on the restriction of $F$ to the strongly connected subgraph $NW_{\mathcal G}$. For many generalized gauge actions there is at most one $\beta$-value for which the three conditions in Theorem \ref{21-03-18c} can be met, and we ask

\begin{q} Is there a strongly connected digraph with a potential function $F$ such that there exists a conservative $\beta$-KMS weight for $\alpha^F$ for more than one value of $\beta$?
\end{q}

When $\mathcal G$ is strongly connected the gauge action on $C^*(\mathcal G)$ has a conservative $\beta$-KMS weight if and only if $\beta$ is the Gurevich entropy of $\mathcal G$ and $\mathcal G$ is recurrent in the sense of Ruette, \cite{Ru}.

\section{The factor type of a conservative $\beta$-KMS weight}

In the setting of Section \ref{sec3}, assume that $NW_{\mathcal G} \neq \emptyset$ and pick a vertex $v$ in $NW_{\mathcal G}$. Then   
$$
 \left\{ \beta F(\mu) - \beta F(\mu'): \ \mu, \mu' \in P_f(\mathcal G), \ r(\mu) =
  r(\mu') = s(\mu) = s(\mu') = v \right\} 
$$
is a subgroup of $\mathbb R$ which does not depend on the vertex $v \in NW_{\mathcal G}$. Let $R_{\mathcal G,F}$ be the closure in $\mathbb
R$ of this subgroup.

\begin{lemma}\label{nov21a} Assume that $\mathcal G$ is cofinal and that $NW_{\mathcal G} \neq \emptyset$. Let $\psi$ be an extremal $\beta$-KMS
  weight for $\alpha^F$. Then $ \pi_{\psi}(C^*(\mathcal G))''$ is a hyperfinite
  factor and
\begin{equation}\label{nov21}
\Gamma\left( \pi_{\psi}(C^*(\mathcal G))''\right) \subseteq R_{\mathcal G,F} \ .
\end{equation}
\end{lemma}
\begin{proof}  $M = \pi_{\psi}(C^*(\mathcal G))''$ is hyperfinite by Lemma \ref{nuclear}. In the following proof we suppress $\pi_{\psi}$ in the notation and consider $C^*(\mathcal G)$ as a subalgebra of $M$. We will show that 
$\mathbb R \backslash R_{\mathcal G,F} \ \subseteq \ \mathbb R \backslash
\Gamma(M)$.
Let therefore $r \in \mathbb R \backslash R_{\mathcal G,F}$ and choose a function $f \in
L^1(\mathbb R)$ whose Fourier transform $\hat{f}$ satisfies that
$\hat{f}(t) = 0$ for all $t \in R_{\mathcal G,F}$ and $\hat{f}(r) \neq
0$. Fix a vertex $v \in NW_{\mathcal G}$ and let $\mu,\nu \in P_f(\mathcal G)$ be finite paths with $s(\mu) = s(\nu) = v$. We assume that $r(\mu) = r(\nu)$ since $S_{\mu}S_{\nu}^*$ is zero otherwise. Since $v $ is wandering and $NW_{\mathcal G}$ is strongly connected there is a finite path $\delta$ in $\mathcal G$ such that $s(\delta) = r(\mu) = r(\nu)$ and $r(\delta) = v$. Then
$$
\beta F(\mu) - \beta F(\nu) = \beta F(\mu\delta) - \beta F(\nu \delta ) \in R_{\mathcal G,F} \ .
$$
It follows that
$$
\theta_f(S_{\mu}S_{\nu}^*) = \int_{\mathbb R} f(t) \theta_t(S_{\mu}S_{\nu}^*) \ dt =
\hat{f}(\beta(F(\mu) - F(\nu))) S_{\mu}S_{\nu}^* = 0 \ 
$$
because $\beta(F(\mu) - F(\nu)) \in R_{\mathcal G,F}$. Since the elements of
the form $S_{\mu}S_{\nu}^*$ for some $\mu, \nu \in P_{f}(\mathcal G)$ with $s(\mu) = s(\nu) = v$ span a strongly dense subspace of $P_vMP_v$ we
conclude that $\theta_f(P_vMP_v) = \{0\}$. Since $\hat{f}(r) \neq 0$
we conclude that $r \notin \Sp(P_vMP_v)$, and hence $r \notin \Gamma(M)$.
\end{proof}

\begin{lemma}\label{28-03-18} Assume $C^*(\mathcal G)$ is simple and that there is a $\beta$-KMS weight for $\alpha^F$. Let $\mu, \nu, \delta$ be finite paths in $\mathcal G$ such that $|\delta| > \max \{|\mu|, |\nu|\}$ and $F(\mu) = F(\nu)$. Then 
$$
P_{\delta}S_{\mu}S_{\nu}^*P_{\delta} = \begin{cases}  0, & \ \text{when} \ \mu \neq \nu \\ P_{\delta} P_{\mu} , & \ \text{when} \ \mu = \nu \ . \end{cases} 
$$ 
\end{lemma}
\begin{proof} If $|\mu| \neq |\nu|$ and $S_{\delta}^*S_{\mu}S_{\nu}^*S_{\delta} \neq 0$, the relations defining $C^*(\mathcal G)$ imply that a piece of $\mu$ or a piece of $\nu$ will be a loop $\kappa$ of positive length in $\mathcal G$ such that $F(\kappa) = 0$. By Lemma 10.6 in \cite{Th5} the existence of a $\beta$-KMS weight rules out the existence of such a loop. It follows that $ S_{\delta}^*S_{\mu}S_{\nu}^*S_{\delta}$ can only be non-zero when $|\mu| = |\nu|$. But $S_{\delta}^*S_{\mu} \neq 0$ implies that $\mu$ must be the initial piece of $\delta$ and similarly $S_{\nu}^*S_{\delta} \neq 0$ implies that $\nu$ must also be the initial piece of $\delta$. Therefore $S_{\delta}^*S_{\mu}S_{\nu}^*S_{\delta} \neq 0$ implies that $\mu = \nu$, in which case $P_{\delta}S_{\mu}S_{\nu}^*P_{\delta} = S_{\delta}S_{\delta}^*S_{\mu}S_{\mu}^*S_{\delta}S_{\delta}^*  = P_{\delta}P_{\mu}$.

\end{proof}

\begin{thm}\label{nov9} Assume that $C^*(\mathcal G)$ is simple and $NW_{\mathcal G}$ not empty. Let $\psi$ be a conservative $\beta$-KMS weight for the generalized gauge action $\alpha^F$. Then
$$
\Gamma( \pi_{\psi}(C^*(\mathcal G))'') = R_{\mathcal G,F} \ .
$$ 
\end{thm}
\begin{proof} The proof is a further development of the proofs of Proposition 4.11 in \cite{Th1} and Theorem 4.1 in \cite{Th4}. As in the proof of Lemma \ref{nov21a} we suppress $\pi_{\psi}$ in the notation and consider $C^*(\mathcal G)$ as a subalgebra of $M$. The modular automorphism group $\theta$ on $M$ defined by $\tilde{\psi}$ is given by
$$
\theta_t = \tilde{\alpha^F}_{-\beta t} \ ,
$$
where $\tilde{\alpha^F}$ is the
$\sigma$-weakly continuous extension of $\alpha^F$. Note that $\beta \neq 0$ by Proposition 2.3 in \cite{Th5}. Let $N \subseteq M$ be the fixed point algebra for $\theta$ and consider a vertex $v \in NW_{\mathcal G}$. By Lemma \ref{nov21} it suffices to show that $R_{\mathcal G,F} 
\subseteq \Gamma(M)$, and from Lemme 2.3.3 and Proposition 2.2.2 in \cite{Co1} we see that it suffices for this to consider a non-zero central projection $q
\in P_vNP_v$ for some vertex $v\in NW_{\mathcal G}$, and show that $\beta F(l)
\subseteq \Sp(qMq)$ when $l$ is a loop in $\mathcal G$ such that $s(l) = r(l) = v$. Let $\omega$ be the state on $P_vMP_v$ given by $\omega(a) =
\psi(P_v)^{-1}\tilde{\psi}(a)$. Then $\omega$ is a faithful normal state which
is a trace on $P_vNP_v$
and we consider the corresponding $2$-norm 
$$
\|a\|_v = \sqrt{\omega(a^*a)}  \ .
$$
By Kaplansky's density theorem there is an element $f \in P_v\mathcal A P_v$
  such that $0 \leq f \leq 1$ and
  $\left\|q- f\right\|_v$ is as small as we want. Then $f$ is a linear combination of elements of the form $S_{\mu}S_{\nu}^*$ where $\mu, \nu \in P_f(\mathcal G)$ and $s(\mu) = s(\nu) = v$. Note that
  $$
 \lim_{R \to \infty}  \frac{1}{R} \int_0^R \theta_{t}(S_{\mu}S_{\nu}^*) \ dt \ = \begin{cases} S_{\mu}S_{\nu}^* & \ \text{when} \ F(\mu) = F(\nu) \\ 0 & \ \text{when} \ F(\mu) \neq F(\nu) \ ,
 \end{cases}
 $$
 with convergence in norm,
 and that
 $$
\left\| q -   \frac{1}{R} \int_0^R \theta_{t}(f) \ dt 
\right\|_v  = \left\|\frac{1}{R} \int_0^R \theta_{t}(q -f) \ dt 
\right\|_v \leq \left\| q-f\right\|_v
$$
by Kadisons Schwarz-inequality. By exchanging $\lim_{R \to \infty}  \frac{1}{R} \int_0^R \theta_{t}(f) \ dt$ for $f$ we may assume that
$$
f = \sum_{i=1}^N \lambda_i S_{\mu_i} S_{\nu_i}^*  \ ,
$$
where $0 \leq \lambda_i \leq 1$ and $F(\mu_i) = F(\nu_i)$ for all $i$.
Let $m_{\psi}$ be the $e^{\beta F}$-conformal measure on $P(\mathcal G)$ defined by $\psi$ and let $m$ be the restriction of the measure $m_{\psi}(Z(v))^{-1}m_{\psi}$ to $Z(v)$. Then $m$ is a Borel probability measure on $Z(v) \subseteq P(\mathcal G)$ such that
$$
\omega \left(P_{\mu}\right) = m(Z(\mu))
$$
for all $\mu \in P_f(\mathcal G)$ with $s(\mu) = v$. For each $k \in \mathbb N$ we let $M_k$ be the set of paths $\delta$ in $\mathcal G$ such that $s(\delta) =v$ and $|\delta| = k$. Then 
$$
\sum_{\delta \in M_k} \omega(P_{\delta}) = \sum_{\delta \in M_k} m(Z(\delta)) = 1 \ ,
$$
which implies that $1 = \sum_{\delta \in M_k} P_{\delta}$, where the sum converges with respect to the 2-norm and hence also in the strong operator topology. It follows that we can define $Q_k : P_vMP_v \to P_vMP_v$ such that
  $$
  Q_k(m) = \sum_{\delta \in M_k} P_{\delta} m   P_{\delta} \ .
  $$
Then $Q_k$ is a positive linear map of norm one and $Q_k(q) = q$. When $k > \max \{|\mu|, |\nu|\}$ it follows from Lemma \ref{28-03-18} that
 $$
 Q_k\left(S_{\mu}S_{\nu}^*\right) = \begin{cases} P_{\mu},  & \ \text{when} \ \mu = \nu \\ 0, & \ \text{when} \ \mu \neq \nu \ . \end{cases}
 $$
 Thus, for some $k$ large enough, we have that $Q_k(f)$ is a linear combination
$$
Q_k(f) = \sum_{i=1}^{N'} \lambda_i P_{\mu_i} \ ,
$$
where $0 \leq \lambda_i \leq 1$ for all $i$. Using Kadisons Schwarz inequality again we find that
$$
\left\| q - Q_k(f)\right\|_v =  \left\| Q_k(q - f)\right\|_v \leq \sqrt{\omega \left( Q_k\left(( q - f)^2\right)\right)} = \|q-f\|_v 
$$
since $\omega \circ Q_k = \omega$ for all $k$. By exchanging $f$ with $Q_k(f)$ for some $k$ large enough we may assume that
\begin{equation}\label{18-03-18}
f =  \sum_{i=1}^{N'} \lambda_i P_{\mu_i} \ .
\end{equation}
Now observe that since $\psi$ is conservative by assumption it follows that $m$ is concentrated on 
$$
\left\{ p \in Z(v) : \ s(p_i) = v \ \ \text{for infinitely many} \ i \right\} \ .
$$
 Fix one of the paths $\mu_i$ appearing in \eqref{18-03-18}. Let $H_j$ denote the set of finite paths $\delta$ of length $j$ such that $s(\delta) = r(\mu_i), \ r(\delta) = v$. Then, up to an $m$-null set,
$$
\bigcup_{j=1}^{\infty} \left\{Z(\mu_i\delta) : \ \delta \in H_j\right\} = Z(\mu_i) \  .
$$
We can therefore find a finite set $K_i \subseteq \bigcup_{j=1}^{\infty} H_j$ such that 
$$
m\left(Z(\mu_i) \backslash \sqcup_{\delta \in K_i} Z(\mu_i \delta)\right) 
$$
is as small as we want. Note that 
$$
\left\| P_{\mu_i} - \sum_{\delta \in K_i} P_{\mu_i \delta} \right\|_v = \sqrt{ m\left(Z(\mu_i) \backslash \sqcup_{\delta \in K_i} Z(\mu_i \delta)\right) }
$$
is then also small. Therefore, by exchanging $\sum_{\delta \in K_i} P_{\mu_i \delta}$ for $P_{\mu_i}$ we may assume that $s(\mu_i) = r(\mu_i) = v$ for all $v$. Finally, since $q$ is a projection, a standard argument, as in the
proof of Lemma 12.2.3 in \cite{KR}, allows us to select a subset of the
$\mu_i$'s and arrange, after a renumbering, that 
\begin{equation}\label{nov17x}
p = \sum_{i=1}^M P_{\mu_i}
\end{equation}
is a projection in $P_v\mathcal AP_v$ such that
\begin{equation}\label{nov8x}
\left\|q - p\right\|_v \leq \epsilon \ ,
\end{equation}
where $\epsilon > 0$ can be chosen as small as we need. We
choose $\epsilon > 0$ so small that
$$
\omega(q) - e^{F(l)\beta}\epsilon - 3\epsilon \ > \ 0 \ .
$$
For each $\mu_i$ from (\ref{nov17x}) we let $w_i \in P_v\mathcal AP_v$ be the elements
$$
w_i = S_{\mu_i}S_{\mu_il}^* \ .
$$
 Each $w_i$ is a partial isometry such that
\begin{enumerate}
\item[a)] $w_i{w_i}^* = P_{\mu_i}$ ,
\item[b)] ${w_i}^* w_i = P_{\mu_il} \leq
  P_{\mu_i}$, and
\item[c)] $\alpha^F_t\left(w_i\right) = e^{-i F(l)t}w_i$
  for all $t \in \mathbb R$.
\end{enumerate}
Set $w = \sum_{i=1}^M w_i$ and note that $w$ is a partial isometry. It follows from b) that $wp = w$ and therefore from (\ref{nov8x}), c), b) and a) that
\begin{equation*}
\begin{split}
& \omega (qwqw^*q) \ \geq \
\omega(wqw^*) - 2\epsilon  \ = \ e^{\beta F(l)} \omega(qw^*w) - 2\epsilon \ \\
& \geq \ e^{\beta F(l)}\omega(pw^*w) - e^{\beta F(l)}\epsilon -2\epsilon \ = e^{\beta F(l)}\omega(w^*w) - e^{\beta F(l)}\epsilon -2\epsilon \ \\
& = \omega(ww^*)  - e^{\beta F(l)}\epsilon -2\epsilon \ = \omega(p)  - e^{\beta F(l)}\epsilon -2\epsilon \ \geq \ \omega(q)  - e^{\beta F(l)}\epsilon -3\epsilon \ .
\end{split}
\end{equation*}
The choice of $\epsilon$ ensures that $\omega (qwqw^*q) > 0$ and hence that $qwq \neq 0$. Since $\theta_{t} (qwq) = e^{-it\beta F(l)}qwq$ for all $t \in \mathbb R$, it follows from Lemme 2.3.6 in \cite{Co1} that $\beta F(l) \in \Sp(qMq)$, as desired. 

\end{proof}

In combination with Corollary \ref{22-03-18} we get the following

\begin{cor}\label{23-03-18}  Assume that $C^*(\mathcal G)$ is simple and that $NW_{\mathcal G} \neq \emptyset$. Let $\varphi$ be a $\beta$-KMS state for the restriction of $\alpha^F$ to $P_vC^*(\mathcal G)P_v$ for some $v \in \mathcal G_V$. When $\sum_{n=0}^{\infty} A(\beta)^n_{w,w} = \infty $ for some (and hence all) $w \in NW_{\mathcal G}$, the Connes invariant of $\pi_{\varphi}(P_vC^*(\mathcal G)P_v)''$ is
$\Gamma\left( \pi_{\varphi}(P_vC^*(\mathcal G)P_v)''\right) \ = \ R_{\mathcal G,F}$.
\end{cor}

In an Appendix in \cite{Th5} it is shown that in the setting of Theorem \ref{nov9} the subgroup $R_{\mathcal G,F}$ is never $\{0\}$. Therefore, thanks to Connes' classification in \cite{Co2} and Haagerup's result in \cite{Ha}, we get the following

\begin{cor}\label{24-03-18} In the setting of Theorem \ref{nov9} and Corollary \ref{23-03-18} the factors $ \pi_{\psi}(C^*(\mathcal G))''$ and $\pi_{\varphi}(P_vC^*(\mathcal G)P_v)''$ are isomorphic; they are both isomorphic to the hyperfinite factor of type $III_{\lambda}$ for some $0 < \lambda \leq 1$.
\end{cor}

\begin{example} (Generalized gauge actions on $O_{\infty}$.) The Cuntz algebra $O_{\infty}$, \cite{Cu}, is the graph $C^*$-algebra $C^*(\mathcal G)$ when $\mathcal G$ is the countable digraph with one vertex and infinitely many arrows, $a_n, \ n =1,2,3, \cdots$. Since $O_{\infty}$ is unital all proper weights are bounded and can be normalized to states.  Let $\{t_n\}_{n=1}^{\infty}$ be a sequence of real numbers and define $F: \mathcal G_{Ar} \to \mathbb R$ such that $F(a_n) = t_n$ for all $n$. The gauge action, where $t_n = 1$ for all $n$, was considered by Olesen and Pedersen who showed in \cite{OP} that there are no KMS states for the gauge action. In general, it follows from \cite{Th3} that a $\beta$-KMS state exists iff $\sum_{n=1}^{\infty} e^{-\beta t_n} \leq 1$ and that it is unique. There is a conservative $\beta_0$-KMS state for $\alpha^F$ iff 
$$
\sum_{n=1}^{\infty} e^{-\beta_0 t_n} = 1 \ .
$$ 
By Theorem \ref{nov9} the Connes invariant of the corresponding factor is the closed subgroup of $\mathbb R$ generated by the numbers $\{\beta_0t_i\}_{i=1}^{\infty}$. It follows that the factor is always of type $III$ and never of type $III_{0}$ while all hyperfinite factors of type $III_{\lambda}$ for $0 < \lambda \leq 1$ can occur by varying the sequence $\{t_n\}$. The KMS states for $\alpha^F$ that are not conservative occur for values of $\beta$ for which $\sum_{n=1}^{\infty} e^{-\beta t_n} < 1 $ and they correspond to minimal proper almost $A(\beta)$-harmonic vectors, albeit of a somewhat degenerate kind since the vectors only have one coordinate.     
\end{example}


\begin{thebibliography}{WWWWW} 







\bibitem[BF]{BF} S. D. Barreto and F. Fidaleo, {\em
On the structure of KMS states of disordered systems},
Comm. Math. Phys. {\bf 250} (2004), 1–-21.
 
\bibitem[BHRS]{BHRS} T. Bates, J. H. Hong, I. Raeburn, W. Szymanski, {\em The ideal structure of the $C^*$-algebras of infinite graphs}, Illinois J. Math. {\bf 46} (2002), 1159--1176.


\bibitem[BC]{BC} J.-B. Bost and A. Connes, {\em Hecke algebras, type III factors and phase transitions with spontaneous symmetry breaking in number theory}, Selecta Math. (N.S.) {\bf 1} (1995), 411-457.

\bibitem[BR]{BR} O. Bratteli and D.W. Robinson, {\em Operator Algebras
    and Quantum Statistical Mechanics I + II}, Texts and Monographs in
  Physics, Springer Verlag, New York, Heidelberg, Berlin, 1979 and 1981.

\bibitem[CPPR]{CPPR} A. Carey, J. Phillips, I. Putnam, A. Rennie, {\em
Families of type III KMS states on a class of $C^*$-algebras containing $O_n$ and $Q_N$}, J. Funct. Anal. {\bf 260} (2011), 1637–-1681. 




\bibitem[C]{C} F. Combes, {\em Poids associ\'e \`a une alg\`ebre
    hilbertienne \`a gauche}, Compos. Math. {\bf 23} (1971), 49-77.

\bibitem[Co1]{Co1} A. Connes, {\em Une classification des facteurs de type III}, Ann. Sci. Ecole Norm. Sup. {\bf 6} (1973), 133-252.

\bibitem[Co2]{Co2} A. Connes, {\em Classification of injective factors. Cases $II_1$, $II_{\infty}$, $III_{\lambda}, \lambda \neq 1$}, Ann. of Math. (2) {\bf 104} (1976), 73–-115.


\bibitem[CE]{CE} M. D. Choi and E. G. Effros, {\em 
Separable nuclear $C^*$-algebras and injectivity}, Duke Math. J. {\bf 43} (1976), 309–-322.

 
\bibitem[CT]{CT}  J. Christensen and K. Thomsen, {\em Diagonality of
    actions and KMS weights}, J. Oper. Th. {\bf 76} (2016), 449-471. 
   
   
\bibitem[Cu]{Cu} J. Cuntz, {\em Simple $C^*$-algebras generated by isometries}, Comm. Math. Phys. {\bf 57} (1977), 173-185.   
   
\bibitem[DT]{DT} D. Drinen and M. Tomforde, {\em The $C^*$-algebras of arbitrary graphs}, Rocky Mountain J. Math. {\bf 35} (2005), 105-135.   
 
 \bibitem[EFW]{EFW} M. Enomoto, M. Fujii and Y. Watatani, {\em KMS
    states for gauge action on $O_A$}, Math. Japon. {\bf 29} (1984), 607-619.

\bibitem[HWH]{HWH} R. Haag, M. Winnink and N.M. Hugenholtz, {\em On the equilibrium states in quantum statistical mechanics}, Comm. Math. Phys., {\bf 5} (1967), 215–-236.


\bibitem[Ha]{Ha} U. Haagerup, {\em Connes' bicentralizer problem and
    uniqueness of the injective factor of type $III_1$}, Acta
  Math. {\bf 158} (1987), 95--148.    
    
\bibitem[I]{I} M. Izumi, {\em The flow of weights and the Cuntz-Pimsner algebras}, Comm. Math. Phys., to appear.

\bibitem[IKW]{IKW} M. Izumi, T. Kajiwara and Y. Watatani, {\em KMS states and branched points}, Ergod. Th. Dyn. Syst. {\bf 27} (2007), 1887-1918.   
  
\bibitem[KW]{KW} T. Kajiwara and Y. Watatani, {\em
KMS states on finite-graph $C^*$-algebras}, Kyushu J. Math. {\bf 67} (2013), 83–-104.   
  
  
\bibitem[KR]{KR} R.V. Kadison and J.R. Ringrose, {\em Fundamentals of
    the Theory of Operator Algebras II}, Academic Press, London 1986.  
    
\bibitem[KP]{KP} A. Kumjian and D. Pask, {\em $C^*$-algebras of directed graphs and group actions},  Ergod. Th. Dyn. Syst. {\bf 19} (1999), 1503–-1519.    
    
\bibitem[Ku]{Ku} J. Kustermans, {\em KMS-weights on $C^*$-algebras},
  arXiv: 9704008v1.



\bibitem[KV1]{KV1} J. Kustermans and S. Vaes, {\em Weight theory for
    $C^*$-algebraic quantum groups}, arXiv:990163.



\bibitem[KV2]{KV2} J. Kustermans and S. Vaes, {\em Locally compact
    quantum groups}, Ann. Scient. \'Ec. Norm. Sup. {\bf 33} (2000), 837-934.


\bibitem[LN]{LN} M. Laca and S. Neshveyev, {\em Type $III_1$ equilibrium states of the Toeplitz algebra of the affine semigroup over the natural numbere}, J. Func. Analysis {\bf 261} (2011), 169-187.

\bibitem[LLNSW]{LLNSW} M. Laca, N. Larsen, S. Neshveyev, A. Sims, S.B.G. Webster, {\em Von Neumann algebras of strongly connected higher-rank graphs}, Math. Ann. {\bf 363} (2015), 657–-678.


\bibitem[N]{N} S. Neshveyev, {\em
von Neumann algebras arising from Bost-Connes type systems}, 
Int. Math. Res. Not. IMRN {\bf 2011} (2011), 217–-236. 

\bibitem[O]{O} R. Okayasu, {\em Type III factors arising from
    Cuntz-Krieger algebras}, Proc. Amer. Math. Soc. {\bf 131} (2003), 2145-2153. 
    
 
\bibitem[OP]{OP} D. Olesen and G.K. Pedersen, {\em Some  $C^*$-dynamical systems with a single KMS-state}, Math. Scand. {\bf 42} (1978), 111-118.   
    
    
\bibitem[Ru]{Ru} S. Ruette, {\em On the Vere-Jones classification and
    existence of maximal measure for countable topological Markov
    chains}, Pac. J. Math. {\bf 209} (2003), 365-380.

\bibitem[Sz]{Sz} W. Szymanski, {\em Simplicity of Cuntz-Krieger
    algebras of infinite matrices}, Pac. J. Math. {\bf 122} (2001), 249-256.

\bibitem[Th1]{Th1} K. Thomsen, {\em KMS weights on groupoid and graph $C^*$-algebras}, J. Func. Analysis {\bf 266} (2014), 2959--2988.






\bibitem[Th2]{Th2} K. Thomsen, {\em Exact circle maps and KMS states}, Israel J. Math. {\bf 205} (2015), 397--420. 

\bibitem[Th3]{Th3} K. Thomsen, {\em KMS weights on graph $C^*$-algebras}, Adv. Math. {\bf 309} (2017), 334–-391.


\bibitem[Th4]{Th4} K. Thomsen, {\em Phase transition in $O_2$}, Comm. Math. Phys. {\bf 349} (2017), 481–-492.


\bibitem[Th5]{Th5} K. Thomsen, {\em KMS weights, conformal measures and ends in digraphs}, arXives 1612.04716 v2.


\bibitem[Y1]{Y1} D. Yang, {\em Type III von Neumann algebras associated with 2-graphs}, Bull. Lond. Math. Soc. {\bf 44} (2012), 675-–686.  



\bibitem[Y2]{Y2} D. Yang, {\em Factoriality and type classification of k-graph von Neumann algebras}, Proc. Edinb. Math. Soc. (2) {\bf 60} (2017), 499–-518. 








\end{thebibliography}
\end{document}